\documentclass[11pt,reqno]{amsart}
\usepackage{amsmath,amsthm,amssymb}

\newtheorem{theorem}{Theorem}
\newtheorem{lemma}[theorem]{Lemma}
\newtheorem{proposition}[theorem]{Proposition}

\theoremstyle{definition}

\newtheorem{definition}[theorem]{Definition}

\newtheorem{remark}[theorem]{Remark}

\begin{document}

\title[On a Theorem of Wang]{On a Theorem of Wang for Complex Homogeneous Manifolds}

\author[D. N. Pham]{David N. Pham}
\address{Department of Mathematics $\&$ Computer Science, QCC CUNY, Bayside, NY 11364}
\email{dnpham@qcc.cuny.edu}



\begin{abstract}
In \cite{Wang1954}, Wang proved (among other things) a sufficiency result for a compact homogeneous manifold $G/H$ to admit a $G$-invariant complex structure.  In this note, we give a new Lie theoretic proof of Wang's theorem which relies on nothing more than the familiar properties of the root space decomposition of a compact Lie group and the maximal tori theorem.  It should be noted that the recent work of Ni and Wallach \cite{NiWallach2025} also revisits the aforementioned theorem of Wang (and others) and offers new Lie theoretic proofs as well.  However, the approach of \cite{NiWallach2025} relies on such objects as Borel subalgebras, parabolic subalgebras, and Iwasawa decomposition which may be somewhat less familiar to the working differential geometer. 
\end{abstract}

\maketitle

\section{Introduction}
In \cite{Wang1954}, Wang studied complex structures on compact homogeneous manifolds and proved a number of important results.  Among those results, he proved the following sufficiency condition for the existence of a homogeneous, that is, invariant, complex structure on a homogeneous manifold:
\begin{theorem}[Wang, 1954]
\label{thmWang}
Let $G$ be a compact semisimple Lie group and $H$ a closed connected subgroup such that the homogeneous manifold $G/H$ is even dimensional.  If the semisimple part of $H$ coincides with the semisimple part of the centralizer of some toral subgroup of $G$, then $G/H$ admits a homogeneous complex structure.
\end{theorem}
\noindent Samelson in \cite{Samelson1953} proved an analogous result for the special case of Lie groups by showing that every compact Lie group of even dimension admits a left-invariant complex structure.  Wang's 1954 paper is a classic in the differential geometry literature.  However, it is not an entirely easy paper to digest.  The clearest sign of this is the recent paper by Ni and Wallach \cite{NiWallach2025} who proved a number of interesting Lie theoretic results and then applied those results to recover the theorem of Wang (as well as others).  For a Lie theorist, the paper of Ni and Wallach uses fairly standard machinery.  However, for the working differential geometer who is not as well versed in Lie theory or representation theory, the machinery of \cite{NiWallach2025} could appear somewhat daunting especially if one is only interested in understanding the above theorem of Wang.  In \cite{NiWallach2025}, one encounters such objects as Borel subalgebras, parabolic subalgebras, and Iwasawa decomposition.  For example, the celebrated books of Kobayashi and Nomizu \cite{KobayashiNomizu1963, KobayashiNomizu1969}, long viewed as the ``bible of differential geometry" make no mention of the aforementioned objects.  Similarly, these objects do not appear in other standard textbooks on differential geometry and complex geometry such as \cite{Lee2013, Huybrechts2005, Voisin2002}.  For this reason, there is certainly some benefit in presenting a new Lie theoretic proof that avoids all mention of these objects, thus making the proof more accessible to the working differential geometer.  

In the current paper, we prove the above theorem of Wang using basic Lie theory. More precisely, we derive the theorem of Wang from scratch using nothing more than the familiar properties of the root space decomposition of a compact Lie group and the maximal tori theorem.  We do not prove any new results.  The paper is purely expository and its main benefit is from the pedagogical perspective.  

The rest of the paper is organized as follows.  In Section \ref{secBackground}, we briefly review the essentials needed for the proof of Theorem \ref{thmWang}, namely homogeneous reductive manifolds and its algebraic description at the Lie algebra level as well as the root space decomposition of a compact Lie group.  In Section \ref{secTheProof}, we derive the theorem of Wang using nothing more than the root space decomposition of a compact Lie group and the maximal tori theorem.  

\section{Preliminaries}
\label{secBackground}
\subsection{Complex structures on Homogeneous Reductive Manifolds}
In this section, we review some standard results at the intersection of homogeneous reductive manifolds and complex geometry.  We do not opt for full generality here.  Our goal is just to review the bare essentials needed for the proof of Theorem \ref{thmWang}.  The reader interested in full proofs and going well beyond what is covered here is referred to the following literature \cite{Arvanitoyeorgos2003, Helgason1978, KobayashiNomizu1969, Nomizu1954, Huybrechts2005, BorelHirzebruch1958}.

Let $G$ be a compact semisimple Lie group and $H$ a connected closed subgroup of $G$.  Let $B$ denote the Killing form of $\mathfrak{g}:=\mbox{Lie}(G)$ which we recall is defined as 
$$
B(X,Y):=\mbox{tr}(\mbox{ad}_X\circ \mbox{ad}_Y),\hspace*{0.2in} \forall~X,Y\in \mathfrak{g},
$$
where $\mbox{ad}_XY:=[X,Y]$.  The Killing form is symmetric and $\mbox{Ad}$-invariant:
$$
B(\mbox{Ad}_g X,\mbox{Ad}_gY)=B(X,Y).
$$
Differentiating shows that $B$ is also $\mbox{ad}$-invariant:
$$
B(\mbox{ad}_XY,Z)+B(Y,\mbox{ad}_XZ)=0,
$$
The condition that $G$ is semisimple is equivalent to the condition that $B$ is nondegenerate.  Actually, since $G$ is also compact, one can show that B is also negative definite:
$$
B(X,X)<0,\qquad \forall~X\in \mathfrak{g}-\{0\}.
$$
Let $\mathfrak{m}\subset \mathfrak{g}$ be the orthogonal complement of $\mathfrak{h}:=\mbox{Lie}(H)$:
$$
\mathfrak{m}:=\mathfrak{h}^\perp:=\{X\in \mathfrak{g}~|~B(X,\mathfrak{h})=0\}.
$$
The decomposition $\mathfrak{g}=\mathfrak{h}\oplus \mathfrak{m}$ satisfies
\begin{equation}
\label{eqReductiveDecomp}
[\mathfrak{h},\mathfrak{m}]\subset \mathfrak{m}
\end{equation}
This is an immediate consequence of the $\mbox{ad}$-invariance of $B$.  Indeed, for $X,Y\in \mathfrak{h}$ and $Z\in \mathfrak{m}$, we have
\begin{align*}
0&=B([X,Y],Z)+B(Y,[X,Z])\\
&=0+B(Y,[X,Z])\\
&=B(Y,[X,Z]),
\end{align*}
where the second equality follows from the fact that $\mathfrak{h}$ is a Lie subalgebra.  Hence, we see that $[X,Z]\in \mathfrak{m}:=\mathfrak{h}^\perp$.  Condition (\ref{eqReductiveDecomp}) means that $G/H$ is a \textit{reductive homogeneous manifold} and the decomposition $\mathfrak{g}=\mathfrak{h}\oplus \mathfrak{m}$ is often called a \textit{reductive decomposition}.  To be somewhat more precise, one says that $G/H$ is a reductive homogeneous manifold if 
\begin{equation}
\label{eqReductiveDecomp2}
\mbox{Ad}_h \mathfrak{m}\subset \mathfrak{m},\qquad \forall~h\in H.
\end{equation}
By differentiating, we see that condition (\ref{eqReductiveDecomp2}) implies (\ref{eqReductiveDecomp}).  However, as we are assuming that $H$ is connected, conditions (\ref{eqReductiveDecomp}) and (\ref{eqReductiveDecomp2}) are actually equivalent in this case.

If we concern ourselves only with homogeneous, that is, $G$-invariant, structures on $G/H$, then everything can be formulated at the Lie algebra level using the reductive decomposition.  This is entirely analogous to restricting one's attention to left-invariant structures on a Lie group and reducing everything to linear algebra.  This is the beauty of homogeneous manifolds.  Everything can be reduced to pure algebra if one is only interested in $G$-invariant structures.

Let $\pi: G\rightarrow G/H$ be the quotient map.  Notation wise, we write 
$$
[g]:=\pi(g)=gH.
$$
Additionally, denote the left $G$-action on $G/H$ by
$$
\tau: G\times G/H\rightarrow G/H,\qquad \tau_g([x]):=\tau(g,[x])=gxH=[gx]
$$
Suppose that $\mathcal{J}$ is a homogeneous almost complex structure on $G/H$.  Then $\mathcal{J}^2=-\mbox{id}$ and 
\begin{equation}
\nonumber
\mathcal{J}\circ d\tau_g=d\tau_g\circ \mathcal{J}.
\end{equation}
The above condition implies that $\mathcal{J}$ is completely determined by its value at the identity coset $[e]$ in much the same way that a left-invariant almost complex structure is determined by its value at the identity element of the group.  Hence, all of the information about $\mathcal{J}$ is encoded in the following linear map:
$$
\mathcal{J}_{[e]}: T_{[e]}(G/H)\rightarrow  T_{[e]}(G/H).
$$
The quotient map $\pi$ together with the reductive decomposition $\mathfrak{g}=\mathfrak{h}\oplus \mathfrak{m}$ gives the natural identification 
$$
d\pi|_{\mathfrak{m}}: \mathfrak{m}\stackrel{\sim}{\longrightarrow} T_{[e]}(G/H).
$$
From $\mathcal{J}$, one obtains a linear map $J:\mathfrak{m}\rightarrow \mathfrak{m}$ which is uniquely defined by
\begin{equation}
\label{eqDefJ}
d\pi \circ J=\mathcal{J}\circ d\pi\big|_{\mathfrak{m}}.
\end{equation}
From the definition  of $J$, one easily sees that 
\begin{equation}
\label{eqJsq}
J^2=-\mbox{id}_{\mathfrak{m}}.
\end{equation}  
The $G$-invariance of $\mathcal{J}$ implies that 
\begin{equation}
\label{eqAdJ}
\mbox{Ad}_h\circ J=J\circ \mbox{Ad}_h,\qquad \forall~h\in H.
\end{equation}
Differentiating, we have
\begin{equation}
\label{eqadJ}
\mbox{ad}_X\circ J=J\circ \mbox{ad}_X,\qquad \forall~X\in \mathfrak{h}.
\end{equation}
Once again, as we are assuming that $H$ is connected, it follows that equations (\ref{eqAdJ}) and (\ref{eqadJ}) are equivalent.   Conversely, one can show that if we start with a linear map $J:\mathfrak{m}\rightarrow \mathfrak{m}$ which satisfies (\ref{eqJsq}) and (\ref{eqadJ}) (for the case where $H$ is connected), one obtains a homogeneous almost complex structure on $G/H$ by defining $\mathcal{J}$ via (\ref{eqDefJ}).  In this way, we have a one to one correspondence between homogeneous almost complex structures on $G/H$ and certain linear maps at the Lie algebra level.  

By the celebrated Newlander-Nirenberg Theorem, $\mathcal{J}$ is integrable, that is, it arises from a complex structure on $G/H$ if and only if its Nijenhuis tensor vanishes: $N_{\mathcal{J}}=0$ where
$$
N_{\mathcal{J}}(X,Y):=\mathcal{J}[\mathcal{J}X,Y]+\mathcal{J}[X,\mathcal{J}Y]+[X,Y]-[\mathcal{J}X,\mathcal{J}Y],\hspace*{0.1in} \forall~X,Y\in \mathfrak{X}(G/H)
$$
With a bit of work (see e.g. Ch. X of \cite{KobayashiNomizu1969}), one can show that $N_{\mathcal{J}}=0$ if and only if $N_J=0$ where
$$
N_J(X,Y):=J[JX,Y]_{\mathfrak{m}}+J[X,JY]_{\mathfrak{m}}+[X,Y]_{\mathfrak{m}}-[JX,JY]_{\mathfrak{m}},\hspace*{0.2in}\forall~X,Y\in \mathfrak{m}.
$$
In the above expression, $[\cdot,\cdot]_{\mathfrak{m}}$ means projection onto $\mathfrak{m}$ in the reductive decomposition.  With some abuse of proper terminology, we call a linear map $J:\mathfrak{m}\rightarrow \mathfrak{m}$ a homogeneous complex structure if it induces an integrable (homogeneous) almost complex structure on $G/H$.

For the proof of Theorem \ref{thmWang}, it is more convenient to view $J:\mathfrak{m}\rightarrow \mathfrak{m}$ in terms of a certain decomposition of the complexified vector space $\mathfrak{m}_{\mathbb{C}}$.  To this end, we make the following definition:
\begin{definition}
\label{defDecompComplex}
A homogeneous complex structure for the reductive decomposition $\mathfrak{g}=\mathfrak{h}\oplus \mathfrak{m}$ is a decomposition 
$$
\mathfrak{m}_{\mathbb{C}}=\mathfrak{m}^+\oplus \mathfrak{m}^-
$$
which satisfies the following conditions:
\begin{itemize}
\item[(1)] $\mathfrak{m}^-=\overline{\mathfrak{m}^+}$ where $\overline{\mathfrak{m}^+}$  denotes the complex conjugate of $\mathfrak{m}^+$
\item[(2)] $[\mathfrak{h},\mathfrak{m}^+]\subset \mathfrak{m}^+$
\item[(3)] $[\mathfrak{m}^+,\mathfrak{m}^+]_{\mathfrak{m}_{\mathbb{C}}}\subset \mathfrak{m}^+$.
\end{itemize}
\end{definition}
\noindent It is an easy exercise to show that Definition \ref{defDecompComplex} is exactly equivalent to a homogeneous complex structure $J:\mathfrak{m}\rightarrow \mathfrak{m}$.  Actually, one can show that conditions (1) and (2) in Definition \ref{defDecompComplex} are equivalent to a homogeneous almost complex structure $J:\mathfrak{m}\rightarrow \mathfrak{m}$ and condition (3) is equivalent to the integrability condition $N_J\equiv 0$.
The precise correspondence between a homogeneous complex structure $J:\mathfrak{m}\rightarrow \mathfrak{m}$ and a decomposition $\mathfrak{m}_{\mathbb{C}}=\mathfrak{m}^+\oplus \mathfrak{m}^-$ satisfying (1)-(3) in Definition \ref{defDecompComplex} is given as follows.  Starting with a homogeneous complex structure $J$, we obtain $\mathfrak{m}^{\pm}$ via
$$
\mathfrak{m}^+=\{X-iJX~|~X\in \mathfrak{m}\},\qquad \mathfrak{m}^-=\{X+iJX~|~X\in \mathfrak{m}\}.
$$
On the other hand, starting with the decomposition $\mathfrak{m}_{\mathbb{C}}=\mathfrak{m}^+\oplus \mathfrak{m}^-$ satisfying (1)-(3), we obtain a homogeneous complex structure $J$ by first defining a linear map $J:\mathfrak{m}_{\mathbb{C}}\rightarrow \mathfrak{m}_{\mathbb{C}}$ via
$$
J\big|_{\mathfrak{m}^+}=i\cdot \mbox{id}_{\mathfrak{m}^+},\qquad J\big|_{\mathfrak{m}^-}=-i\cdot \mbox{id}_{\mathfrak{m}^-}.
$$ 
The associated homogeneous complex structure is then obtained by restricting $J$ to $\mathfrak{m}$.  Note that condition (1) of Definition \ref{defDecompComplex} ensures that $J$ is real, that is, $J\mathfrak{m}\subset \mathfrak{m}$.

\subsection{The Root Space Decomposition}
In this section, we briefly review the root space decomposition of a compact Lie group.  The reader who may be somewhat unfamiliar with this topic and would like to delve deeper is referred to the following standard references \cite{DuistermaatKolk2000, Hall2015, Helgason1978}.

Let $G$ be a compact semisimple Lie group and let $\mathfrak{t}$ be a maximal abelian Lie subalgebra of $\mathfrak{g}:=\mbox{Lie}(G)$.  Then the complexified Lie algebra $\mathfrak{g}_{\mathbb{C}}$ can always be decomposed in the following form:
$$
\mathfrak{g}_{\mathbb{C}}=\mathfrak{t}_{\mathbb{C}}\oplus \bigoplus_{\alpha\in \Delta_{\mathfrak{g}}}\mathfrak{g}_{\alpha}.
$$
In the above decomposition, an element $\alpha\in \Delta_{\mathfrak{g}}$ is called a \textit{root} and $\mathfrak{g}_\alpha$ is the corresponding \textit{root space}.  More plainly, $\alpha$ is a $\mathbb{C}$-linear map 
$$
\alpha:\mathfrak{t}_{\mathbb{C}}\rightarrow \mathbb{C}
$$
which satisfies 
$$
[X,E_\alpha]=\alpha(X)E_\alpha,\qquad \forall ~X\in \mathfrak{t},~E_\alpha\in \mathfrak{g}_\alpha.
$$
Once again, we let $B$ denote the Killing form of $\mathfrak{g}$ which is extended by $\mathbb{C}$-bilinearity to $\mathfrak{g}_{\mathbb{C}}$.  The root space decomposition satisfies the following properties:
\begin{itemize}
\item[(i)] $-\Delta_{\mathfrak{g}}=\Delta_{\mathfrak{g}}$,
\item[(ii)] $\dim \mathfrak{g}_\alpha=1$, 
\item[(iii)] $[\mathfrak{g}_\alpha,\mathfrak{g}_\beta]= \mathfrak{g}_{\alpha+\beta}$ if $\alpha+\beta\in \Delta_{\mathfrak{g}}$; $[\mathfrak{g}_{\alpha},\mathfrak{g}_{-\alpha}]\subset \mathfrak{t}_{\mathbb{C}}$; and $[\mathfrak{g}_\alpha,\mathfrak{g}_\beta]=0$ if $\alpha+\beta\notin \Delta_{\mathfrak{g}}\cup \{0\}$,
\item[(iv)] $B|_{\mathfrak{g}_{\alpha}\times \mathfrak{g}_{\beta}}$ is nondegenerate for $\alpha+\beta=0$ and is zero otherwise,
\item[(v)] $B(\mathfrak{t},\mathfrak{g}_\alpha)=0$,
\item[(vi)] $\mathfrak{g}_{-\alpha}=\overline{\mathfrak{g}_{\alpha}}$, and
\item[(vii)] $\alpha|_{\mathfrak{t}}$ is imaginary valued
\end{itemize}
for all $\alpha,\beta\in \Delta_{\mathfrak{g}}$.   Also, since $B$ is negative definite for $G$ compact and semisimple, we also have that $B|_{\mathfrak{t}}$ is also negative definite and hence nondegenerate on $\mathfrak{t}$.   

If $G$ is compact and not necessarily semisimple, one still has a root space decomposition where (i)-(vii) still holds.  This follows from the fact that the Lie algebra of any compact Lie group always decomposes as 
$$
\mathfrak{g}=Z_{\mathfrak{g}}(\mathfrak{g})\oplus \mathfrak{g}',
$$
where $Z_{\mathfrak{g}}(\mathfrak{g})$ is the centralizer of $\mathfrak{g}$ in $\mathfrak{g}$ and $\mathfrak{g}':=[\mathfrak{g},\mathfrak{g}]$ can be shown to be semisimple.  The following observation is simple but important nonetheless:
\begin{proposition}
\label{propMaximalTorus}
Let $\mathfrak{t}$ be any maximal abelian Lie subalgebra of $\mathfrak{g}$.  Then
\begin{itemize}
\item[(a)] $\mathfrak{t}_{\mathfrak{g}'}:=\mathfrak{t}\cap \mathfrak{g}'$ is a maximal abelian Lie subalgebra of $\mathfrak{g}'$
\item[(b)] $\mathfrak{t}=Z_{\mathfrak{g}}(\mathfrak{g})\oplus \mathfrak{t}_{\mathfrak{g}'}$.
\end{itemize}
\end{proposition}
\begin{proof}
 From the decomposition $\mathfrak{g}= Z_{\mathfrak{g}}(\mathfrak{g})\oplus \mathfrak{g}'$, we have that $ Z_{\mathfrak{g}}(\mathfrak{g})\cap \mathfrak{t}_{\mathfrak{g}'}=\{0\}$.  
 Let $\mathfrak{a}:=Z_{\mathfrak{g}}(\mathfrak{g})\oplus \mathfrak{t}_{\mathfrak{g}'}$.  Since  $Z_{\mathfrak{g}}(\mathfrak{g})$ commutes with everything in $\mathfrak{g}$ and $\mathfrak{t}$ is a maximal abelian Lie subalgebra of $\mathfrak{g}$, it follows immediately that $Z_{\mathfrak{g}}(\mathfrak{g})\subset \mathfrak{t}$.  From the definition of $\mathfrak{t}_{\mathfrak{g}'}$, we clearly have $\mathfrak{a}\subset \mathfrak{t}$.  Let $X\in \mathfrak{t}$ and decompose $X$ as $X=X_1+X_2$  where $X_1\in Z_{\mathfrak{g}}(\mathfrak{g})\subset \mathfrak{t}$ and $X_2\in \mathfrak{g}'$. Then 
 $$
 X_2=X-X_1\in \mathfrak{t}\cap \mathfrak{g}'=:\mathfrak{t}_{\mathfrak{g}'},
 $$ 
which implies that $X\in \mathfrak{a}$.  Hence, $\mathfrak{a}\supset\mathfrak{t}$ which proves (b).  For (a), note that if $\mathfrak{t}_{\mathfrak{g}'}$ is not a maximal abelian Lie subalgebra of $\mathfrak{g}'$, then we could find a strictly larger abelian Lie subalgebra $\widetilde{\mathfrak{t}}_{\mathfrak{g}'}$ of $\mathfrak{g}'$ such that $\mathfrak{t}_{\mathfrak{g}'}\subset \widetilde{\mathfrak{t}}_{\mathfrak{g}'}$.  Then 
$$
\widetilde{\mathfrak{t}}:= Z_{\mathfrak{g}}(\mathfrak{g})\oplus \widetilde{\mathfrak{t}}_{\mathfrak{g}'}
$$
is a strictly larger abelian Lie subalgebra which contains $\mathfrak{t}$, but this would contradict the maximality of $\mathfrak{t}$.  Hence, $\mathfrak{t}_{\mathfrak{g}'}$ must be a  maximal abelian Lie subalgebra of $\mathfrak{g}'$. 
\end{proof}
\noindent Consequently, if we let $\mathfrak{t}$ be any maximal abelian Lie subalgebra of $\mathfrak{g}$ and let $\mathfrak{t}_{\mathfrak{g}'}:=\mathfrak{t}\cap \mathfrak{g}'$, the root space decomposition of $\mathfrak{g}'_{\mathbb{C}}$ with respect to $\mathfrak{t}_{\mathfrak{g}'}$ gives rise to a root space decomposition for $\mathfrak{g}_{\mathbb{C}}$:
$$
\mathfrak{g}_{\mathbb{C}}=\underbrace{Z_{\mathfrak{g}}(\mathfrak{g})_{\mathbb{C}}\oplus (\mathfrak{t}_{\mathfrak{g}'})_{\mathbb{C}}}_{\mathfrak{t}_{\mathbb{C}}}\oplus \bigoplus_{\alpha\in \Delta_{\mathfrak{g}'}}\mathfrak{g}'_{\alpha}.
$$
Strictly speaking, $\alpha\in \Delta_{\mathfrak{g}'}$ is a $\mathbb{C}$-linear map $\alpha:(\mathfrak{t}_{\mathfrak{g}'})_{\mathbb{C}}\rightarrow \mathbb{C}$.  In the present context, it is understood that $\alpha$ is extended to all of $\mathfrak{t}_{\mathbb{C}}$ by defining 
\begin{equation}
\label{eqAlphaZExt}
\alpha\big|_{Z_{\mathfrak{g}}(\mathfrak{g})}\equiv 0.
\end{equation}
This must be the case since for $X\in Z_{\mathfrak{g}}(\mathfrak{g})$, we have $[X,\cdot]=0$.  So for $E_\alpha\in \mathfrak{g}_{\alpha}$, $\alpha\in \Delta_{\mathfrak{g}'}$, we have
$$
[X,E_\alpha]=\alpha(X)E_\alpha,\qquad \forall~X\in  Z_{\mathfrak{g}}(\mathfrak{g})
$$
only if one adopts (\ref{eqAlphaZExt}).  The only condition that fails for the case when $Z_{\mathfrak{g}}(\mathfrak{g})\neq \{0\}$ is that the restriction of the Killing form $B$ to $\mathfrak{t}$ fails to be nondegenerate.  

Critical to the proof of Theorem \ref{thmWang} is the choice of a system of \textit{positive roots}.  We  now recall how a system of positive roots is chosen.  A system of positive roots is a decomposition 
$$
\Delta_{\mathfrak{g}}=\Delta_{\mathfrak{g}}^+\cup \Delta_{\mathfrak{g}}^-
$$
such that 
\begin{itemize}
\item[(i)] $\Delta_{\mathfrak{g}}^-=-\Delta_{\mathfrak{g}}^+$,
\item[(ii)] if $\alpha,\beta\in \Delta_{\mathfrak{g}}^+$ and $\alpha+\beta$ is a root, then $\alpha+\beta\in \Delta^+_{\mathfrak{g}}$.
\end{itemize}
$\Delta_{\mathfrak{g}}^+$ are the set of \textit{positive roots} and $\Delta_{\mathfrak{g}}^-$ are the set of \textit{negative roots}.  Of course, one can also designate $\Delta_{\mathfrak{g}}^-$ the set of positive roots and  $\Delta_{\mathfrak{g}}^+$ the set of negative roots and conditions (i) and (ii) would still hold.  One constructs a decomposition of $\Delta_{\mathfrak{g}}$ into positive and negative roots simply by choosing an element 
$$
X_0\in \mathfrak{t}-\bigcup_{\alpha\in \Delta_{\mathfrak{g}}}\ker \alpha\neq \emptyset
$$
and then defining 
$$
\Delta^+_{\mathfrak{g}}:=\{\alpha\in \Delta_{\mathfrak{g}}~|~ -i\alpha(X_0)>0\},\qquad \Delta^-_{\mathfrak{g}}:=\{\alpha\in \Delta_{\mathfrak{g}}~|~ -i\alpha(X_0)<0\}
$$
where we recall that the restriction of $\alpha\in \Delta_{\mathfrak{g}}$ to $\mathfrak{t}\subset \mathfrak{t}_{\mathbb{C}}$ is imaginary valued.  From the above construction, it is easy to see that one has a decomposition  $\Delta_{\mathfrak{g}}=\Delta_{\mathfrak{g}}^+\cup \Delta_{\mathfrak{g}}^-$ which satisfies conditions (i) and (ii) for a system of positive roots.  As we will see shortly, the choice of $X_0\in \mathfrak{t}$ that determines the decomposition into positive and negative roots cannot be chosen arbitrarily.

\section{The Proof of Theorem \ref{thmWang}}
\label{secTheProof}
\noindent We recall the hypotheses of Theorem \ref{thmWang}.  First, $G$ is a compact semisimple Lie group and $H$ is a connected closed subgroup.  Let $\mathfrak{g}:=\mbox{Lie}(G)$ and $\mathfrak{h}:=\mbox{Lie}(H)$.   We define $\mathfrak{m}:=\mathfrak{h}^\perp$ to be the orthogonal complement of $\mathfrak{h}$ with respect to the Killing form $B$ of $\mathfrak{g}$.  This gives a reductive decomposition 
$$
\mathfrak{g}=\mathfrak{h}\oplus \mathfrak{m}.
$$
In addition, Theorem \ref{thmWang} also assumes the semisimple part of $H$ coincides with the semisimple part of the centralizer of some toral subgroup $S$ of $G$.  Let 
$$
Z_G(S):=\{g\in G|~gs=sg\}=\bigcap_{s\in S}Z_G(s)
$$
be the centralizer of $S$ in $G$.  Note that $Z_G(S)$ is a closed subgroup of $G$ and hence a compact Lie group in its own right.  Also, we have
$$
\mbox{Lie}(Z_G(S))=Z_{\mathfrak{g}}(\mathfrak{s}):=\{X\in \mathfrak{g}~|~[X,\mathfrak{s}]=0\}
$$  
At the Lie algebra level, the hypothesis of Theorem \ref{thmWang} is the condition
\begin{equation}
\label{eqThmHypothesis}
[\mathfrak{h},\mathfrak{h}]=[Z_{\mathfrak{g}}(\mathfrak{s}),Z_{\mathfrak{g}}(\mathfrak{s})].
\end{equation}
For ease of notation, we set 
$$
\mathfrak{b}:=Z_{\mathfrak{g}}(\mathfrak{s})
$$
and
$$
\mathfrak{h}':=[\mathfrak{h},\mathfrak{h}],\qquad \mathfrak{b}':=[\mathfrak{b},\mathfrak{b}].
$$
The hypothesis of Theorem \ref{thmWang} can now be written more succinctly as 
\begin{equation}
\mathfrak{h}'=\mathfrak{b}'.
\end{equation}
Let $\mathfrak{t}$ be a maximal abelian Lie subalgebra of $\mathfrak{g}$ which contains $\mathfrak{s}$.  Since $[\mathfrak{s},\mathfrak{t}]=0$, the definition of $\mathfrak{b}$ implies
\begin{equation}
\label{eqMaximalTorusInB}
\mathfrak{s}\subset \mathfrak{t}\subset \mathfrak{b}.
\end{equation}
Hence, $\mathfrak{b}$ contains a maximal abelian Lie subalgebra of $\mathfrak{g}$.  By Proposition \ref{propMaximalTorus}, 
\begin{equation}
\mathfrak{t}_{\mathfrak{h}'}:=\mathfrak{t}\cap \mathfrak{h}'=\mathfrak{t}\cap \mathfrak{b}'
\end{equation}
is a maximal abelian Lie subalgebra of $\mathfrak{b}'$ and 
\begin{equation}
\label{eqMaximalTorusDecomp}
\mathfrak{t}=Z_{\mathfrak{b}}(\mathfrak{b})\oplus \mathfrak{t}_{\mathfrak{h}'}.
\end{equation}
Taking the root space decomposition of $\mathfrak{h}'_{\mathbb{C}}=\mathfrak{b}'_{\mathbb{C}}$ with respect to $\mathfrak{t}_{\mathfrak{h}'}$, we arrive at the following decomposition for $\mathfrak{b}_{\mathbb{C}}$:
\begin{equation}
\label{eqBCdecomp}
\mathfrak{b}_{\mathbb{C}}=\underbrace{Z_{\mathfrak{b}}(\mathfrak{b})_{\mathbb{C}}\oplus (\mathfrak{t}_{\mathfrak{h}'})_{\mathbb{C}}}_{\mathfrak{t}_{\mathbb{C}}}\oplus\bigoplus_{\alpha\in \Delta_{\mathfrak{h}'}}\mathfrak{h}'_{\alpha}.
\end{equation}
Since $\mathfrak{t}$ is also a maximal abelian Lie subalgebra of $\mathfrak{g}$, the decomposition in (\ref{eqBCdecomp}) extends to the root space decomposition of $\mathfrak{g}_{\mathbb{C}}$ by including the remaining root spaces associated to $\mathfrak{t}$ in $\mathfrak{g}_{\mathbb{C}}$:
\begin{equation}
\label{eqGCdecomp}
\mathfrak{g}_{\mathbb{C}}=\underbrace{Z_{\mathfrak{b}}(\mathfrak{b})_{\mathbb{C}}\oplus (\mathfrak{t}_{\mathfrak{h}'})_{\mathbb{C}}}_{\mathfrak{t}_{\mathbb{C}}}\oplus\bigoplus_{\alpha\in \Delta_{\mathfrak{h}'}}\mathfrak{h}'_{\alpha}\oplus \bigoplus_{\alpha \in \Lambda}\mathfrak{g}_\alpha.
\end{equation}
The set of roots appearing in the root space decomposition of $\mathfrak{g}_{\mathbb{C}}$ with respect to $\mathfrak{t}$ is then
$$
\Delta_{\mathfrak{g}}:=\Delta_{\mathfrak{h}'}\cup \Lambda.
$$
\noindent We now show that one can use $\mathfrak{s}$ to construct a new abelian Lie algebra $\widetilde{\mathfrak{s}}$ which also satisfies the semisimple condition of Theorem \ref{thmWang} while also satisfying $Z_{\mathfrak{h}}(\mathfrak{h})\subset \widetilde{\mathfrak{t}}$ where $\widetilde{\mathfrak{t}}$ is a maximal abelian Lie subalgebra of $\mathfrak{g}$ which contains $\widetilde{\mathfrak{s}}$.
\begin{lemma}
\label{lemA}
$Z_{\mathfrak{b}}(\mathfrak{b})$ is a maximal abelian Lie subalgebra of $Z_{\mathfrak{g}}(\mathfrak{h}')=Z_{\mathfrak{g}}(\mathfrak{b}')$.
\end{lemma}
\begin{proof}
Note that since $\mathfrak{s}$ is abelian and $\mathfrak{b}:=Z_{\mathfrak{g}}(\mathfrak{s})$, we have
$$
\mathfrak{s}\subset Z_{\mathfrak{b}}(\mathfrak{b})\subset Z_{\mathfrak{g}}(\mathfrak{b}').
$$
Let $X\in Z_{\mathfrak{g}}(\mathfrak{b}')$ such that $[X,Z_{\mathfrak{b}}(\mathfrak{b})]=0$.  Since $[X,\mathfrak{b}']=0$ and $\mathfrak{b}=Z_{\mathfrak{b}}(\mathfrak{b})\oplus \mathfrak{b}'$, we have $[X,\mathfrak{b}]=0$.  Hence, $X\in Z_{\mathfrak{g}}(\mathfrak{b})\subset Z_{\mathfrak{g}}(\mathfrak{s})=:\mathfrak{b}$.  Decomposing $X$ as $X=X_1+X_2$ with $X_1\in Z_{\mathfrak{b}}(\mathfrak{b})$ and $X_2\in \mathfrak{b}'$, we have
$$
[X_2,\mathfrak{b}']=[X,\mathfrak{b}']-[X_1,\mathfrak{b}']=0-0=0.
$$
Hence, $X_2\in Z_{\mathfrak{b}'}(\mathfrak{b}')=\{0\}$ since $\mathfrak{b}'$ is semisimple.  Consequently, $X=X_1\in Z_{\mathfrak{b}}(\mathfrak{b})$, which proves the lemma.
\end{proof}
 As $\mathfrak{h}'\subset \mathfrak{h}$, note that $Z_{\mathfrak{h}}(\mathfrak{h})\subset Z_{\mathfrak{g}}(\mathfrak{h}')=Z_{\mathfrak{g}}(\mathfrak{b}')$.  Now we fix a maximal abelian Lie subalgebra  $\mathfrak{u}$ of $Z_{\mathfrak{g}}(\mathfrak{b}')$ such that 
\begin{equation}
Z_{\mathfrak{h}}(\mathfrak{h})\subset \mathfrak{u}.
\end{equation}
Note that as $\mathfrak{b}'=\mathfrak{h}'$ is semisimple, there exists a connected, closed (and hence compact) subgroup $H'$ whose Lie algebra is $\mathfrak{b}'=\mathfrak{h}'$.  We define 
$$
Z_G(H'):=\{g\in G~|~gx=xg~\forall~x\in H'\}.
$$
Note that as $Z_G(H')$ is closed, $Z_G(H')$ is also a compact Lie subgroup of $G$.   With $H'$ connected, the Lie algebra of $Z_G(H')$ is then 
$$
\mbox{Lie}(Z_G(H'))=Z_{\mathfrak{g}}(\mathfrak{b}')=Z_{\mathfrak{g}}(\mathfrak{h}').
$$
\begin{lemma}
\label{lemB}
There exists $g\in Z_G(H')$ such that $\mbox{Ad}_g (Z_{\mathfrak{b}}(\mathfrak{b}))=\mathfrak{u}$.
\end{lemma}
\begin{proof}
$Z_{\mathfrak{b}}(\mathfrak{b})$ and $\mathfrak{u}$ are both maximal abelian Lie subalgebras of the Lie algebra $Z_{\mathfrak{g}}(\mathfrak{b}')=\mbox{Lie}(Z_G(H'))$.  Let $T_1$ and $T_2$ be the maximal tori of $Z_G(H')$ whose Lie algebras are $Z_{\mathfrak{b}}(\mathfrak{b})$ and $\mathfrak{u}$ respectively.  Then by maximal tori theorem, $T_1$ and $T_2$ are conjugates of one another.  Hence, for some $g\in Z_G(H')$, we have 
$$
gT_1g^{-1}=T_2.
$$
Differentiating gives  $\mbox{Ad}_g (Z_{\mathfrak{b}}(\mathfrak{b}))=\mathfrak{u}$.
\end{proof}
\begin{proposition}
\label{propC}
Let $g\in Z_G(H')$ be chosen as in Lemma \ref{lemB}.  Define the following:
\begin{itemize}
\item[(a)]$\widetilde{\mathfrak{s}}:=\mbox{Ad}_g \mathfrak{s}$,
\item[(b)]  $\widetilde{\mathfrak{b}}:=Z_{\mathfrak{g}}(\widetilde{\mathfrak{s}})$,
\item[(c)] $\widetilde{\mathfrak{t}}:=\mathrm{Ad}_g\mathfrak{t}$.
\end{itemize}
Then 
\begin{itemize}
\item[(1)] $\widetilde{\mathfrak{b}}':=[\widetilde{\mathfrak{b}},\widetilde{\mathfrak{b}}]=\mathfrak{h}'$
\item[(2)] $\widetilde{\mathfrak{t}}=\mathfrak{u}\oplus \mathfrak{t}_{\mathfrak{h}'}\supset Z_{\mathfrak{h}}(\mathfrak{h})$
\end{itemize}
\end{proposition}
\begin{proof}
(1): Since $\mbox{Ad}_{g}:\mathfrak{g}\stackrel{\sim}{\rightarrow} \mathfrak{g}$ is a Lie algebra isomorphism, we have 
$$
\mbox{Ad}_g\mathfrak{b}=\widetilde{\mathfrak{b}}.
$$
Since $g\in Z_G(H')$, we have $gxg^{-1}=x$ for all $x\in H'$.  With $\mbox{Lie}(H')=\mathfrak{h}'$, it follows by differentiating that  
$$
\mbox{Ad}_g\big|_{\mathfrak{h}'}=\mbox{id}_{\mathfrak{h}'}.
$$
Hence, 
\begin{align*}
\widetilde{\mathfrak{b}}'&:=[\widetilde{\mathfrak{b}},\widetilde{\mathfrak{b}}]=[\mbox{Ad}_g\mathfrak{b},\mbox{Ad}_g\mathfrak{b}]=\mbox{Ad}_g[\mathfrak{b},\mathfrak{b}]=\mbox{Ad}_g\mathfrak{b}'=\mbox{Ad}_g\mathfrak{h}' =\mathfrak{h}'.
\end{align*}
\noindent (2): From Lemmas \ref{lemA} and \ref{lemB}, $Z_{\mathfrak{b}}(\mathfrak{b})$ and $\mathfrak{u}$ are both maximal abelian Lie subalgebras of $Z_{\mathfrak{g}}(\mathfrak{h}')$ where $\mathfrak{u}=\mbox{Ad}_gZ_{\mathfrak{b}}(\mathfrak{b})$ and $Z_{\mathfrak{h}}(\mathfrak{h})\subset \mathfrak{u}$ (from the definition of $\mathfrak{u}$).  From the definition of $\widetilde{\mathfrak{t}}$ and (\ref{eqMaximalTorusDecomp}), we have
\begin{align*}
\widetilde{\mathfrak{t}}&=\mbox{Ad}_g Z_{\mathfrak{b}}(\mathfrak{b})\oplus \mbox{Ad}_g \mathfrak{t}_{\mathfrak{h}'}\\
&=\mathfrak{u}\oplus \mathfrak{t}_{\mathfrak{h}'}\\
&\supset Z_{\mathfrak{h}}(\mathfrak{h}).
\end{align*}
Note that if $X\in \mathfrak{u}\cap \mathfrak{t}_{\mathfrak{h}'}$, then
$$
X=\mbox{Ad}_{g^{-1}}X\in Z_{\mathfrak{b}}(\mathfrak{b}) \cap \mathfrak{h}'=\{0\}
$$
as $\mathfrak{b}=Z_{\mathfrak{b}}(\mathfrak{b})\oplus \mathfrak{b}'=Z_{\mathfrak{b}}(\mathfrak{b})\oplus \mathfrak{h}'$.  Hence, we indeed have $\mathfrak{u}\cap \mathfrak{t}_{\mathfrak{h}'}=\{0\}$.
\end{proof}
\noindent In light of Proposition \ref{propC}, we will assume from this point forth that the abelian Lie algebra $\mathfrak{s}$ from the hypothesis of Theorem \ref{thmWang} satisfies properties (1) and (2) of Proposition \ref{propC}. In other words, we now have the Lie algebra $\mathfrak{s}$ of a toral subgroup of $G$ such that for 
$$
\mathfrak{b}:=Z_{\mathfrak{g}}(\mathfrak{s})
$$
and $\mathfrak{t}$ a maximal abelian Lie subalgebra containing $\mathfrak{s}$, one has
$$
\mathfrak{s}\subset \mathfrak{t}\subset \mathfrak{b},\qquad \mathfrak{b}':=[\mathfrak{b},\mathfrak{b}]=\mathfrak{h}',\qquad \mathfrak{t}_{\mathfrak{h}'}:=\mathfrak{t}\cap \mathfrak{h}',
$$
and
$$
\mathfrak{t}=Z_{\mathfrak{b}}(\mathfrak{b})\oplus \mathfrak{t}_{\mathfrak{h}'}\supset Z_{\mathfrak{h}}(\mathfrak{h}).
$$

\noindent Now let
\begin{equation}
\mathfrak{t}_{\mathfrak{h}}:=Z_{\mathfrak{h}}(\mathfrak{h})\oplus \mathfrak{t}_{\mathfrak{h}'}\subset \mathfrak{t}.
\end{equation}
Define 
\begin{equation}
\mathfrak{a}:=\{X\in \mathfrak{t}~|~B(X,\mathfrak{t}_{\mathfrak{h}})=0\}.
\end{equation}
Then
\begin{equation}
\label{eqFinalT}
\mathfrak{t}=\mathfrak{a}\oplus \mathfrak{t}_{\mathfrak{h}}.
\end{equation}
The root space decomposition of $\mathfrak{g}_{\mathbb{C}}$ with respect to $\mathfrak{t}$ becomes
\begin{equation}
\label{eqGCdecomp2}
\mathfrak{g}_{\mathbb{C}}=\mathfrak{a}_{\mathbb{C}}\oplus \underbrace{Z_{\mathfrak{h}}(\mathfrak{h})_{\mathbb{C}}\oplus (\mathfrak{t}_{\mathfrak{h}'})_{\mathbb{C}}\oplus\bigoplus_{\alpha\in \Delta_{\mathfrak{h}'}}\mathfrak{h}'_{\alpha}}_{\mathfrak{h}_{\mathbb{C}}}\oplus \bigoplus_{\alpha \in \Lambda}\mathfrak{g}_\alpha.
\end{equation}
Moreover, note that
\begin{align}
\nonumber
\mathfrak{b}_{\mathbb{C}}&=Z_{\mathfrak{b}}(\mathfrak{b})_{\mathbb{C}}\oplus \mathfrak{b}'_{\mathbb{C}}=Z_{\mathfrak{b}}(\mathfrak{b})_{\mathbb{C}}\oplus \mathfrak{h}'_{\mathbb{C}}=Z_{\mathfrak{b}}(\mathfrak{b})_{\mathbb{C}}\oplus (\mathfrak{t}_{\mathfrak{h}'})_{\mathbb{C}}\oplus \bigoplus_{\alpha\in \Delta_{\mathfrak{h}'}}\mathfrak{h}'_{\alpha}\\
\nonumber
&=\mathfrak{t}_{\mathbb{C}}\oplus\bigoplus_{\alpha\in \Delta_{\mathfrak{h}'}}\mathfrak{h}'_{\alpha}=\mathfrak{a}_{\mathbb{C}}\oplus( \mathfrak{t}_{\mathfrak{h}})_{\mathbb{C}}\oplus \bigoplus_{\alpha\in \Delta_{\mathfrak{h}'}}\mathfrak{h}'_{\alpha}\\
\label{eqFinalB}
&=\mathfrak{a}_{\mathbb{C}}\oplus \mathfrak{h}_{\mathbb{C}},
\end{align} 
where the second line follows from  Proposition \ref{propC} and the definition of $\mathfrak{a}$.

From the definition of $\mathfrak{a}$, we immediately have
$$
B(\mathfrak{a},Z_{\mathfrak{h}}(\mathfrak{h}))=B(\mathfrak{a},\mathfrak{t}_{\mathfrak{h}'})=0.
$$
From (\ref{eqFinalT}) and the properties of the root space decomposition, we have
$$
B(\mathfrak{a},\mathfrak{h}_\alpha')=0,\qquad B(\mathfrak{g}_\beta,Z_{\mathfrak{h}}(\mathfrak{h}))=B(\mathfrak{g}_\beta,\mathfrak{t}_{\mathfrak{h}'})=0\qquad \forall~\alpha\in \Delta_{\mathfrak{h}'},~\beta\in \Lambda.
$$
Since  $\alpha+\beta\neq 0$ for all $\alpha\in \Delta_{\mathfrak{h}'},~\beta\in \Lambda$ (where again $\Delta_{\mathfrak{g}}=\Delta_{\mathfrak{h}'}\cup \Lambda$), we also have
$$
B(\mathfrak{h}'_{\alpha},\mathfrak{g}_\beta)=0\qquad \forall ~\alpha\in \Delta_{\mathfrak{h}'},~\beta\in \Lambda.
$$
From this, we conclude that
\begin{equation}
\label{eqFinalM}
\mathfrak{m}_{\mathbb{C}}:=({\mathfrak{h}}^\perp)_{\mathbb{C}}=\mathfrak{a}_{\mathbb{C}}\oplus  \bigoplus_{\alpha \in \Lambda}\mathfrak{g}_\alpha.
\end{equation}
To arrive at a decomposition 
$$
\mathfrak{m}_{\mathbb{C}}=\mathfrak{m}^+\oplus \mathfrak{m}^-
$$
which is a homogeneous complex structure for the reductive decomposition $\mathfrak{g}=\mathfrak{h}\oplus \mathfrak{m}$ (see Definition \ref{defDecompComplex}), we need to choose a suitable system of positive roots for $\Delta_{\mathfrak{g}}=\Delta_{\mathfrak{h}'}\cup \Lambda$.  First, we give a simple (yet important characterization) of the roots which lie in $\Delta_{\mathfrak{h}'}$.  
\begin{proposition}
\label{propRootResult}
Let $\alpha\in \Delta_{\mathfrak{g}}$.  Then 
\begin{itemize}
\item[(1)] $\alpha\in \Delta_{\mathfrak{h}'}$ if and only if $\alpha(\mathfrak{s})=0$. 
\item[(2)] $\alpha\in \Lambda$ if and only if $\alpha(X)\neq 0$ for some $X\in \mathfrak{s}$. 
\end{itemize}
\end{proposition}
\begin{proof}
(1): Let $\alpha\in \Delta_{\mathfrak{h}'}$ and let $X\in \mathfrak{s}$ and $E_\alpha\in \mathfrak{h}'_\alpha$.  Since 
$$
\mathfrak{h}'_\alpha\subset \mathfrak{b}_{\mathbb{C}}:=Z_{\mathfrak{g}}(\mathfrak{s}),
$$
we have
$$
0=[X,E_\alpha]=\alpha(X)E_\alpha.
$$  
Hence, $\alpha(X)=0$.  

Now let $\alpha\in \Delta_{\mathfrak{g}}$ such that $\alpha(\mathfrak{s})=0$.  Suppose $\alpha\in \Lambda$.  Let $X\in \mathfrak{s}$ and $F_\alpha\in \mathfrak{g}_\alpha$.  Then
$$
[X,F_\alpha]=\alpha(X)F_\alpha=0.
$$
Hence, $[\mathfrak{s},\mathfrak{g}_\alpha]=0$ which implies $\mathfrak{g}_\alpha\subset \mathfrak{b}_{\mathbb{C}}$ which is a contradiction from (\ref{eqGCdecomp2}) and (\ref{eqFinalB}).  As $\Delta_{\mathfrak{g}}=\Delta_{\mathfrak{h}'}\cup \Lambda$, we conclude that $\alpha\in \Delta_{\mathfrak{h}'}$.
\vspace*{0.1in}\\
\noindent (2): This follows immediately from (1).
\end{proof}
\noindent We now define a decomposition of $\Delta_{\mathfrak{g}}$ into positive and negative roots as follows.  Since $\alpha(\mathfrak{s})\neq 0$ for $\alpha\in \Lambda$ by Proposition \ref{propRootResult}, we may choose an element
$$
X_s\in \mathfrak{s}-\bigcup_{\alpha\in \Lambda}\ker \alpha\neq \emptyset.
$$
Likewise, since $\alpha(\mathfrak{t}_{\mathfrak{h}'})\neq 0$ for $\alpha\in \Delta_{\mathfrak{h}'}$, we may also choose an element
$$
X_h\in \mathfrak{t}_{\mathfrak{h}'}-\bigcup_{\alpha\in \Delta_{\mathfrak{h}'}}\ker \alpha\neq \emptyset.
$$
Let
$$
X_0:=X_s+\epsilon X_h,
$$
where $\epsilon>0$.  Then
$$
\alpha(X_0)=\epsilon\alpha(X_h)\neq 0\qquad\forall~\alpha\in \Delta_{{\mathfrak{h}'}}
$$
Since $\alpha(X_s)\neq 0$ for all $\alpha\in \Lambda$, we can choose $\epsilon>0$ sufficiently small so that  
$$
\alpha(X_0)=\alpha(X_s)+\epsilon\alpha(X_h)\neq 0\qquad \forall~\alpha\in \Lambda.
$$
In fact, for all $\epsilon>0$ sufficiently small, $-i\alpha(X_0)$ and $-i\alpha(X_s)$ carry the same sign for all $\alpha\in \Lambda$ (where we recall that roots are imaginary valued when restricted to $\mathfrak{t}\subset \mathfrak{t}_{\mathbb{C}}$).  We will assume the aforementioned condition from this point forth.  From the above considerations, we have $\alpha(X_0)\neq 0$ for all $\alpha\in \Delta_{\mathfrak{g}}=\Delta_{{\mathfrak{h}'}}\cup \Lambda$.  Hence, $X_0$ determines a decomposition of $\Delta_{\mathfrak{g}}$ into positive and negative roots:  
$$
\Delta^+_{\mathfrak{g}}:=\{\alpha~|~ -i\alpha(X_0)>0\},\qquad \Delta^-_{\mathfrak{g}}:=\{\alpha~|~ -i\alpha(X_0)<0\}.
$$
We also set
$$
\Delta_{{\mathfrak{h}'}}^{\pm}:=\Delta^{\pm}_{\mathfrak{g}}\cap \Delta_{{\mathfrak{h}'}}\qquad \Lambda^{\pm}:=\Delta^{\pm}_{\mathfrak{g}}\cap \Lambda.
$$
By taking $\epsilon>0$ even smaller, we can ensure that the following condition is also satisfied:
\begin{equation}
\label{eqEpsilonCondition}
-i\alpha(X_s)>-i\epsilon[\beta(X_h)-\alpha(X_h)]\qquad \forall~\alpha\in \Lambda^+,~\beta\in \Delta_{\mathfrak{h}'}^+
\end{equation}
The reason for this condition will become clear shortly.  
\begin{proposition}
\label{propRootSystemA}
Let $\alpha\in \Lambda^+$ and $\beta\in \Delta_{\mathfrak{h}'}^+$. 
\begin{itemize}
\item[(i)] If $\alpha+\beta$ is a root, then $\alpha+\beta\in \Lambda^+$.
\item[(ii)] If $\alpha-\beta$ is a root, then $\alpha-\beta\in \Lambda^+$.
\end{itemize}
\end{proposition}
\begin{proof}
(i): Suppose $\alpha+\beta$ is a root.  Then we have either $[\mathfrak{h}'_\beta,\mathfrak{g}_\alpha]=\mathfrak{h}_{\alpha+\beta}'$ or $[\mathfrak{h}'_\beta,\mathfrak{g}_\alpha]=\mathfrak{g}_{\alpha+\beta}$ given the decomposition (\ref{eqGCdecomp2}).  However, since $\mathfrak{g}=\mathfrak{h}\oplus \mathfrak{m}$ is a reductive decomposition, we have $[\mathfrak{h},\mathfrak{m}]\subset \mathfrak{m}$ which implies that $[\mathfrak{h}'_\beta,\mathfrak{g}_\alpha]= \mathfrak{g}_{\alpha+\beta}$.  Moreover, since $\alpha$ and $\beta$ are positive roots, it follows that $\alpha+\beta \in \Lambda^+\subset \Lambda$.  
\vspace*{0.1in}\\
\noindent (ii): Now suppose $\alpha-\beta$ is a root.  Once again, since $[\mathfrak{h},\mathfrak{m}]\subset \mathfrak{m}$, it follows that  $[\mathfrak{h}'_{-\beta},\mathfrak{g}_\alpha]=\mathfrak{g}_{\alpha-\beta}$.  Then
\begin{align*}
-i(\alpha-\beta)(X_0)&=-i\alpha(X_s)-i\epsilon\alpha(X_h)+i\epsilon\beta(X_h)\\
&=-i\alpha(X_s)+i\epsilon[\beta(X_h)-\alpha(X_h)]\\
&>0
\end{align*}
by (\ref{eqEpsilonCondition}).  Hence, $\alpha-\beta\in \Lambda^+\subset \Lambda$.  
\end{proof}
\noindent At long last, we define our decomposition of $\mathfrak{m}_{\mathbb{C}}$ as follows.  First, for $\mathfrak{a}_{\mathbb{C}}$ in (\ref{eqFinalM}), choose \textit{any} decomposition of the form
\begin{equation}
\label{eqAdecomp}
\mathfrak{a}_{\mathbb{C}}=\mathfrak{a}^+\oplus \mathfrak{a}^-,\qquad \mathfrak{a}^-=\overline{\mathfrak{a}^+}.
\end{equation}
Note that from the hypothesis of Theorem \ref{thmWang}, we have that $\dim \mathfrak{m}=\dim G/H$ is even.  Since 
$$
\bigoplus_{\alpha\in \Lambda}\mathfrak{g}_\alpha
$$
is even dimensional as $\Lambda=\Lambda^+\cup \Lambda^-$, $\Lambda^-=-\Lambda^+$, and $\dim \mathfrak{g}_\alpha=1$, it follows that $\mathfrak{a}_{\mathbb{C}}$ and hence $\mathfrak{a}$ is also even dimensional. (After all, the dimension of the complexification $\mathfrak{a}_{\mathbb{C}}$ as a complex vector space is exactly the same as the dimension of $\mathfrak{a}$ as a real vector space.)  Hence, a decomposition of the form given by (\ref{eqAdecomp}) is always possible.  Now let
\begin{equation}
\mathfrak{m}^+_1:=\bigoplus_{\alpha\in \Lambda^+}\mathfrak{g}_{\alpha}\qquad \mathfrak{m}^-_1:=\bigoplus_{\alpha\in \Lambda^-}\mathfrak{g}_{\alpha}.
\end{equation}
Then the homogeneous complex structure for the reductive decomposition $\mathfrak{g}=\mathfrak{h}\oplus \mathfrak{m}$ is given by $\mathfrak{m}_{\mathbb{C}}=\mathfrak{m}^+\oplus \mathfrak{m}^-$ where 
\begin{equation}
\label{eqFinalMdecompA}
\mathfrak{m}^{\pm}=\mathfrak{a}^{\pm}\oplus \mathfrak{m}^{\pm}_1.
\end{equation}
We now quickly verify conditions (1)-(3) of Definition \ref{defDecompComplex}.  Since $\mathfrak{g}_{-\alpha}=\overline{\mathfrak{g}_\alpha}$, we immediately have
$$
\mathfrak{m}^-=\overline{\mathfrak{m}^+}
$$
which verifies condition (1) of  Definition \ref{defDecompComplex}.  Since $\mathfrak{t}_{\mathfrak{h}}:=Z_{\mathfrak{h}}(\mathfrak{h})\oplus \mathfrak{t}_{\mathfrak{h}'}\subset \mathfrak{t}$ and $\mathfrak{a}\subset \mathfrak{t}$, we have 
$$
[\mathfrak{t}_{\mathfrak{h}},\mathfrak{a}]=0,\qquad [\mathfrak{t}_{\mathfrak{h}},\mathfrak{g}_\alpha]\subset \mathfrak{g}_{\alpha}.
$$
Also for $\alpha\in \Lambda^+$ and $\beta\in \Delta_{\mathfrak{h}'}^+$, Proposition \ref{propRootSystemA} implies
$$
[\mathfrak{h}_{\beta}',\mathfrak{g}_\alpha]\subset \mathfrak{m}_1^+\subset \mathfrak{m}^+
$$
and
$$
[\mathfrak{h}_{-\beta}',\mathfrak{g}_\alpha]\subset \mathfrak{m}_1^+\subset \mathfrak{m}^+.
$$
Now let $\alpha\in \Delta_{\mathfrak{h}'}$ and $X_\alpha\in \mathfrak{h}'_{\alpha}$ be a nonzero element.  For $Y\in \mathfrak{a}\subset \mathfrak{t}$, we have
\begin{align*}
0&=B([X_\alpha,Y],\overline{X_{\alpha}})+B(Y,[X_{\alpha},\overline{X_{\alpha}}])\\
&=-\alpha(Y)B(X_\alpha,\overline{X_{\alpha}}),
\end{align*}
where the second equality follows from the fact that $[X_{\alpha},\overline{X_{\alpha}}]\in (\mathfrak{t}_{\mathfrak{h}'})_{\mathbb{C}}\subset (\mathfrak{t}_{\mathfrak{h}})_{\mathbb{C}}$ and $B(\mathfrak{a},\mathfrak{t}_{\mathfrak{h}})=0$.  Since $B$ is negative definite on $\mathfrak{g}$, we have $B(X_\alpha,\overline{X_{\alpha}})<0$.  Hence, $\alpha(Y)=0$.  As $Y\in \mathfrak{a}$ is arbitrary, it follows that $[\mathfrak{a},\mathfrak{h}_{\alpha}']=0$.  This, together with the previous bracket relations, establishes  $[\mathfrak{h},\mathfrak{m}^+]\subset \mathfrak{m}^+$ which confirms condition (2) of Definition \ref{defDecompComplex}.  Lastly, since $\mathfrak{a}^+\subset \mathfrak{a}_{\mathbb{C}}\subset \mathfrak{t}_{\mathbb{C}}$, observe that
$$
[\mathfrak{a}^+,\mathfrak{a}^+]=0,\qquad [\mathfrak{a}^+,\mathfrak{m}^+_1]\subset \mathfrak{m}^+_1.
$$
Also, since $\mathfrak{m}^+_1$ is spanned by the positive root spaces corresponding to $\Lambda^+$, it follows that $[\mathfrak{m}^+_1,\mathfrak{m}^+_1]$ must lie in the span of the positive root space of $\mathfrak{g}_{\mathbb{C}}$: 
$$
[\mathfrak{m}^+_1,\mathfrak{m}^+_1]\subset \mathfrak{m}^+_1\oplus \bigoplus_{\alpha\in \Delta_{\mathfrak{h}'}^+}\mathfrak{h}'_\alpha.
$$
In particular, we have
$$
[\mathfrak{m}^+_1,\mathfrak{m}^+_1]_{\mathfrak{m}_{\mathbb{C}}}\subset \mathfrak{m}^+_1.
$$
From the above considerations, it follows that 
$$
[\mathfrak{m}^+,\mathfrak{m}^+]_{\mathfrak{m}_{\mathbb{C}}}\subset \mathfrak{m}^+,
$$
which confirms condition (3) of Definition \ref{defDecompComplex} and completes the proof.
\vspace*{0.1in}\\
\noindent  As a nice side remark, we recall a result of Borel in \cite{Borel1953}:
\begin{proposition}[Borel, 1953]
\label{propBorel}
Let $G$ be a compact Lie group (not necessarily semisimple) and let $T$ be a maximal torus of $G$, then $G/T$ always admits a homogeneous complex structure.
\end{proposition}
\begin{proof}
Let $\mathfrak{t}:=\mbox{Lie}(T)$ and $\mathfrak{g}:=\mbox{Lie}(G)$.  Let $\Delta$ denote the roots associated to $\mathfrak{t}$ and choose any decomposition of $\Delta$ into positive and negative roots: $\Delta=\Delta^+\cup \Delta^-$. We take the root space decomposition of $\mathfrak{g}_{\mathbb{C}}$ with respect to $\mathfrak{t}$:
$$
\mathfrak{g}_{\mathbb{C}}=\mathfrak{t}_{\mathbb{C}}\oplus \bigoplus_{\alpha\in \Delta^+}\mathfrak{g}_\alpha\oplus \bigoplus_{\alpha\in \Delta^-}\mathfrak{g}_\alpha.
$$
Let 
$$
\mathfrak{m}^{\pm}:=\bigoplus_{\alpha\in \Delta^{\pm}}\mathfrak{g}_\alpha
$$
and define
$$
\mathfrak{m}_{\mathbb{C}}=\mathfrak{m}^+\oplus \mathfrak{m}^-.
$$
Recall that $G$ fails to be semisimple precisely when $\mathfrak{g}$ has a nontrivial center.  In this case, taking the $B$-orthogonal complement of $\mathfrak{t}$ (as was done in the proof of Theorem \ref{thmWang}) does not yield a complement of $\mathfrak{t}$ in $\mathfrak{g}$.   As a consequence of this, $\mathfrak{m}$ is defined directly from the root spaces.  On the other hand, when $G$ is semisimple, the complexification of the $B$-orthogonal complement of $\mathfrak{t}$ yields precisely the above $\mathfrak{m}_{\mathbb{C}}$.

Since $\mathfrak{m}^-=\overline{\mathfrak{m}^+}$ by construction, it follows that $\mathfrak{m}_{\mathbb{C}}$ is given by the complexification of an underlying real space.  Define
$$
\mathfrak{m}:=\mathrm{Re}\left(\mathfrak{m}_{\mathbb{C}}\right).
$$
Then  
$$
\mathfrak{g}=\mathfrak{t}\oplus \mathfrak{m}.
$$
Since $[\mathfrak{t},\mathfrak{m}_{\mathbb{C}}]\subset \mathfrak{m}_{\mathbb{C}}$ as $\mathfrak{m}_{\mathbb{C}}$ is spanned by the root spaces, we also have $[\mathfrak{t},\mathfrak{m}]\subset \mathfrak{m}$.  Hence, the above decomposition is reductive.  

Now, as the Lie algebra of the isotropy group is the maximal abelian Lie algebra $\mathfrak{t}$, condition (2) of Definition \ref{defDecompComplex} is immediate: $[\mathfrak{t},\mathfrak{m}^+]\subset \mathfrak{m}^+$.  Also, as $\mathfrak{m}^+$ is spanned by \textit{all} the positive root spaces, we also immediately have 
$$
[\mathfrak{m}^+,\mathfrak{m}^+]_{\mathfrak{m}_{\mathbb{C}}}=[\mathfrak{m}^+,\mathfrak{m}^+]\subset \mathfrak{m}^+.
$$
Hence, condition (3) of Definition \ref{defDecompComplex} is certainly satisfied and we conclude that $\mathfrak{m}^{\pm}$ determines a homogeneous complex structure for the reductive decomposition $\mathfrak{g}=\mathfrak{t}\oplus \mathfrak{m}$.
\end{proof}
\begin{remark}
The complex homogeneous manifold $G/T$ of Proposition \ref{propBorel} is also special in another regard.  If $G$ is also semisimple (in addition to being compact), then Borel showed in \cite{Borel1954} that the above homogeneous complex structure on $G/T$ can be equipped with a family of homogeneous K\"{a}hler metrics.
\end{remark}



\begin{thebibliography}{99}

\bibitem{Arvanitoyeorgos2003}
A. Arvanitoyeorgos, \textit{An Introduction to Lie Groups and the Geometry 
of Homogeneous Spaces}, Student Mathematical Library, vol. 22, American 
Mathematical Society, Providence, RI, 2003.

\bibitem{Borel1953}
A. Borel, \textit{Sur la cohomologie des espaces fibr\'es principaux 
et des espaces homog\`enes de groupes de Lie compacts}, Ann. of Math. 
(2) 57 (1953), 115--207.

\bibitem{Borel1954}
A. Borel, \textit{K\"ahlerian coset spaces of semisimple Lie groups}, 
Proc. Nat. Acad. Sci. U.S.A. 40 (1954), 1147--1151.

\bibitem{BorelHirzebruch1958}
A. Borel, F. Hirzebruch, \textit{Characteristic classes and homogeneous 
spaces, I}, Amer. J. Math. 80 (1958), 458--538.

\bibitem{DuistermaatKolk2000}
J. J. Duistermaat, J. A. C. Kolk, \textit{Lie Groups}, Universitext, 
Springer-Verlag, Berlin, 2000.

\bibitem{Hall2015}
B. C. Hall, \textit{Lie Groups, Lie Algebras, and Representations: 
An Elementary Introduction}, Second Edition, Graduate Texts in 
Mathematics, vol. 222, Springer, Cham, 2015.

\bibitem{Helgason1978}
S. Helgason, \textit{Differential Geometry, Lie Groups, and Symmetric 
Spaces}, Academic Press, New York, 1978.

\bibitem{Huybrechts2005}
D. Huybrechts, \textit{Complex Geometry: An Introduction}, Universitext, 
Springer-Verlag, Berlin, 2005.

\bibitem{KobayashiNomizu1963}
S. Kobayashi, K. Nomizu, \textit{Foundations of Differential Geometry, 
Volume I}, Interscience Publishers, New York, 1963.

\bibitem{KobayashiNomizu1969}
S. Kobayashi, K. Nomizu, \textit{Foundations of Differential Geometry, 
Volume II}, Interscience Publishers, New York, 1969. 
(Chapters IX--X: Complex manifolds; Chapter X: Homogeneous spaces.)

\bibitem{Lee2013}
J. M. Lee, \textit{Introduction to Smooth Manifolds}, Second Edition, 
Graduate Texts in Mathematics, vol. 218, Springer, New York, 2013.

\bibitem{NiWallach2025}
L. Ni, N. Wallach, \textit{Compact, connected, complex manifolds that 
admit a compact transitive group of holomorphic automorphisms}, 
Acta Math. Sinica, English Series (2025), 1--15.

\bibitem{Nomizu1954}
K. Nomizu, \textit{Invariant affine connections on homogeneous spaces}, 
Amer. J. Math. 76 (1954), 33--65.

\bibitem{Samelson1953}
H. Samelson, \textit{A class of complex-analytic manifolds}, 
Portugal. Math. 12 (1953), 129--132.

\bibitem{Voisin2002}
C. Voisin, \textit{Hodge Theory and Complex Algebraic Geometry, I}, 
Cambridge Studies in Advanced Mathematics, vol. 76, Cambridge 
University Press, Cambridge, 2002.

\bibitem{Wang1954}
H.-C. Wang, \textit{Closed manifolds with homogeneous complex structure}, 
Amer. J. Math. 76 (1954), 1--32.







%
%




\end{thebibliography}
\end{document}